\documentclass[10pt,A4paper]{ article}

\usepackage{amsfonts}
\usepackage{enumerate}
\usepackage{amsmath,amsthm}
\usepackage[all,arc]{xy}
\usepackage{mathrsfs}
\usepackage[pdftex]{graphicx}
\usepackage{multicol}

\newtheorem{thm}{Theorem}[section]
\newtheorem{cor}[thm]{Corollary}
\newtheorem{prop}[thm]{Proposition}
\newtheorem{lem}[thm]{Lemma}

\theoremstyle{definition}
\newtheorem{defn}[thm]{Definition}
\newtheorem{defns}[thm]{Definitions}

\newcommand{\Lip}[0]{\mathrm{Lip}}

\newcommand{\Int}[0]{\mathrm{Int \hspace{3pt}}}
\newcommand{\Stab}[1]{\ensuremath{W_\varepsilon^s(#1)}}
\newcommand{\Unstab}[1]{\ensuremath{W_\varepsilon^u(#1)}}
\newcommand{\Hol}[0]{H\"{o}lder }

\title {On Ruelle's Lemma and Ruelle Zeta Functions}
\author{Paul Wright} 
\date{\today}

\begin{document}
\maketitle

\begin{abstract}

In this article we prove an important inequality regarding the Ruelle operator in hyperbolic flows. This was already proven briefly by Mark Pollicott and Richard Sharp in a low dimensional case \cite{Pollicott-Sharp}, but we present a clearer proof of the inequality, filling in gaps and explaining the ideas in more detail, and extend the inequality to higher dimensional flows. This inequality is necessary to prove a proposition about the analyticity of Ruelle zeta functions.\footnote[1]{This article represents the main result of the author's honours dissertation at the University of Western Australia, supervised by Luchezar Stoyanov.}

\end{abstract}

\section{Introduction}

This article gives a comprehensive and rigourous proof of a lemma by Mark Pollicott and Richard Sharp (called ``Lemma 2'' in \cite{Pollicott-Sharp}). They claim that the result was essentially proved in a paper by David Ruelle \cite{Ruelle}, but Ruelle's paper does not state or prove the lemma explicitly. Another proof of the lemma can be found in a paper by Fr\'{e}d\'{e}ric Naud \cite{Naud}, which clarifies some issues with the proof in \cite{Pollicott-Sharp} but contains a new error. A similar but weaker result was also proven in \cite{Pollicott-Sharp 2}. The three papers \cite{Naud, Pollicott-Sharp, Pollicott-Sharp 2} each prove the lemma briefly for low dimensional flows but each proof has significant gaps and even errors. In this article the proof follows Naud's approach more closely than the original Pollicott-Sharp proof. We recover the original result from \cite{Pollicott-Sharp}, but fill in gaps in the proof and extend the result to higher dimensional flows.


Let M be a compact smooth manifold and let $\phi_t: M \rightarrow M$ be an Axiom A flow \cite{Parry-Pollicott}. Then $M$ contains a non-wandering set that can be decomposed into basic sets $\Lambda_i \subset M$ on which $\phi_t$ is hyperbolic and transitive, its periodic points are dense and there exists a neighbourhood $U \supset \Lambda_i$ with $\Lambda_i = \displaystyle \bigcap^\infty_{n=-\infty} f^n(U)$. See \cite{Parry-Pollicott} for full definitions. We use the fact that the periodic points are dense in $\Lambda$ near the end of the paper. We focus on a single basic set $\Lambda$. 

For any $x \in M$ let $W^s_\varepsilon(x), W^u_\varepsilon(x)$ be the local stable and unstable manifolds through $x$ respectively, as defined in \cite{Parry-Pollicott}. We call a subset $A$ of a stable or unstable manifold \textit{admissible} if $A = \overline{\Int A}$, where the closure and interior are in the induced topology of the local stable or unstable manifold intersected with $\Lambda$. A \textit{Markov rectangle} $R_i$ can be constructed from a point $z_i \in \Lambda$ and admissible subsets $U_i \subset W_\varepsilon^u(z_i) \cap \Lambda$ and $S_i \subset W_\varepsilon^s(z_i) \cap \Lambda$ using a local product structure on $\Lambda$ \cite{Dolgopyat 1, Parry-Pollicott}. The interior of a rectangle $\Int R_i$ can be similarly constructed from $\Int U_i$ and $\Int S_i$. A Markov family of rectangles $R_1, \ldots, R_k \subset \Lambda$ can be constructed as described in \cite{Dolgopyat 1, Parry-Pollicott}, and we let $R = \bigcup_{i=1}^k R_i$. 


In \cite{Pollicott-Sharp}, the stable and unstable leaves are always one-dimensional and each $U_i$ is called an interval. We are working with higher dimensional flows, so the unstable leaves could in general have a higher dimension. We call each $U_i$ a \textit{Markov leaf}. The set of Markov leaves is disjoint. We denote $U = \bigcup_{i=1}^k U_i$ and $\Int U = \bigcup_{i=1}^k \Int U_i$.
 
Define the Poincar\'{e} map $H: R \rightarrow R$ by $H(x) = \phi_{r(x)}(x) \in R$, where $r(x)>0$ is the \textit{first return time}, the smallest positive time such that $\phi_{r(x)}(x) \in R$. 

Define a $k \times k$ matrix $A$ by

\begin{displaymath}
   A(i,j) = \left\{
     \begin{array}{lr}
       1  : H(\Int R_i) \cap \Int R_j \neq \emptyset \\
       0  : \mbox{otherwise.}
     \end{array}
   \right.
\end{displaymath}

We use this matrix to define the symbol spaces $X_A$ and $X_A^+$
$$X_A = \{\xi = \{ \xi_n \}^\infty_{n=-\infty}: \xi_n \in \{1, \ldots , k \}, A(\xi_n, \xi_{n+1}) = 1, \forall n \in \mathbb{Z}\}$$
$$X^+_A = \{x = \{ \xi_n \}^\infty_{n=0 \hspace{9 pt}}: \xi_n \in \{1, \ldots , k \}, A(\xi_n, \xi_{n+1}) = 1, \forall n \geq 0 \}.$$

It is always possible to construct the Markov family such that the matrix $A$ is \textit{irreducible} and \textit{aperiodic} \cite{Parry-Pollicott}.

Define the \textit{stable holonomy map} $\pi^u : R \rightarrow U$ by $\pi^u(x) = W_\varepsilon^s(x) \cap U_i$. Define the expanding map $f: U \rightarrow U$ by $f = \pi^u \circ H|_U$. This map is expanding in the sense that there exist constants $0 < \gamma < 1$ and $C_1>0$ such that, if $f^j(x)$ and $f^j(y)$ are on the same Markov leaf $U_{i_j}$ for every $0 \leq j \leq m$, we have 
\begin{equation} \label{expanding map}
d(x,y) \leq C_1 \gamma^m d(f^m x,f^m y).
\end{equation}

\section{Preliminaries}

\subsection{Cylinders}

\begin{defn}
For each $n \geq 0$, consider all strings of length $n$ of symbols $1, 2, \ldots, k$. We call a string $\alpha = (\alpha_0,\ldots, \alpha_{n-1})$ a \textit{word} if it is admissible under $A$, that is, if $A(\alpha_j,\alpha_{j+1}) = 1$ for all $0 \leq j \leq n-2$. For any word $\alpha$ we write $|\alpha| = n$ and define a subset of one of the Markov leaves $$U_\alpha = U_{\alpha_0} \cap f^{-1} U_{\alpha_1} \cap \ldots \cap f^{-(n-1)} U_{\alpha_{n-1}}.$$
Here $f^{-1}$ denotes the preimage. We call this subset a \textit{cylinder} of length $n$ in the leaf $U_{\alpha_0}$. Each $U_i$ is a cylinder of length $1$ and is always a compact set, but other cylinders can be open, closed, or neither. If $|\alpha| = |\beta| > 1$ then either $\alpha = \beta$ or $U_\alpha \cap U_\beta = \emptyset$. This is not true in general for the closures $\overline{U_\alpha}$ and $\overline{U_\beta}$. A cylinder $U_\alpha$ is always nonempty, and as the length of a string approaches infinity, the corresponding cylinder approaches a single point. Thus a single point in $U$ can be represented by an infinite string of symbols, i.e. an element of $X^+_A$.
\end{defn}

\begin{defns}
Let $\alpha = (\alpha_0, \ldots, \alpha_{n-1})$ be a word of length $n$ and let $i =1, \ldots, k$. Then if $A(\alpha_{n-1},i) = 1$ and $A(i,\alpha_0)=1$, we define
$$\alpha i = (\alpha_0, \ldots, \alpha_{n-1}, i), \hspace{8pt} i \alpha = (i, \alpha_0, \ldots, \alpha_{n-1}), \mbox{ and } \overline{\alpha} = (\alpha_0, \ldots, \alpha_{n-2}).$$

In general the expanding map $f^n$ is not always injective for $n \geq 1$. For each permissible word $\alpha$ we define an inverse that follows the word backwards. This is always possible for points on the interior of a leaf. Let $\alpha$ be a word of length $n$ and let $x \in f(U_{\alpha_{n-1}}) \cap \Int U_i$ for any $i$ with $A(\alpha_{n-1}, i)=1$. Define $f_\alpha^{-1}(x)$ to be the unique point $y$ such that $f^n(y) = x$ and $y \in U_\alpha$. For a single symbol $i$ we can write $ix = f_i^{-1}(x)$. 
\label{word inverse}
\end{defns}

\subsection{\Hol functions}

\begin{defn}
For any $A \subseteq U$, denote by $C(A)$ the set of all continuous functions $w: A \rightarrow \mathbb{C}$, and let $w \in C(A)$. For any $0< \mu \leq 1$, we say that $w$ is $\mu$-\textit{\Hol} on $A$, or \Hol continuous on $A$ with exponent $\mu$, if there exists $C_2 >0$ such that $|w(x) - w(y)| \leq C_2 d(x,y)^\mu$ for all $x,y \in A \cap U_i$, $i = 1 , \ldots , k$. We define the \textit{$\mu$ norm} of $w$ on $A$ by
$$|w|_\mu = \displaystyle \sup \left\{ \frac{|w(x) - w(y)|}{d(x,y)^\mu} : x,y \in A \cap U_i, i \in \{1 , \ldots , k\}, x \neq y \right\}.$$

When $\mu = 1$ we say $w$ is \textit{Lipschitz} or \textit{Lipschitz continuous}. For any admissible set $A$, we denote by $H^\mu(A)$ the set of all functions that are $\mu$-\Hol on $A$. That is,
$$H^\mu(A) = \{w \in C(U): |w|_\mu < \infty \mbox{ on } A\}.$$


We also define the infinity norm $|w|_\infty = \displaystyle\sup_{x \in U} |w(x)|$ and a total norm
$$\|w\|_\mu = |w|_\mu + |w|_\infty.$$

For an operator $L: C(U) \rightarrow C(U)$ we use the operator norm 
$$\|L\|_\mu = \sup\{ \frac{\|L w\|_\mu}{\|w\|_\mu} : w \in C(U), w \neq 0 \}.$$
\end{defn}
Pollicott and Sharp use a different ($C^1$) norm in their paper \cite{Pollicott-Sharp}. 

If $f: A \rightarrow \mathbb{C}$ is $\mu$-\Hol on some set $A$ in any metric space, then there exists a unique extension ext$(f): \overline{A} \rightarrow \mathbb{C}$ of $f$ to the closure of $A$, such that ext$(f)$ is $\mu$-\Hol on $\overline{A}$ \cite{Lipschitzextension}. We call ext$(f)$ the \textit{\Hol extension} of $f$. 

\begin{lem} \label{exp w}
If $w$ is $\mu$-\Hol on a set containing $x$ and $y$, then
\begin{align*}
|e^{w(x) - w(y)} - 1| & \leq |w(x) - w(y)| e^{|w(x) - w(y)|}\\
											& \leq |w|_{\mu} e^{|w|_{\mu} |U|^{\mu}} d(x,y)^\mu.
\end{align*}

\end{lem}

\subsection{Characteristic Function}

For a word $\alpha$ of length $n$ we let $\chi_\alpha$ denote the characteristic function

\begin{displaymath}
\chi_\alpha(x) = \chi_{U_\alpha}(x) = \left\{ 
	\begin{array}{lr}
       1 & \mbox{if } x \in U_\alpha,  \\
       0 & \mbox{otherwise.}
	\end{array}
\right.
\end{displaymath}

Observe that for any word $\alpha$ of length $n$, $\chi_\alpha(x)$ is not Lipschitz on all $U_i$, but it is piecewise Lipschitz. In particular, it is Lipschitz (and hence $\mu$-H\"{o}lder) on $U_\beta$ for any word $\beta$ with $|\beta| \geq |\alpha|$, with Lipschitz and \Hol constants $|\chi_\alpha(x)|_\Lip = |\chi_\alpha(x)|_\mu = 0$. We write
$$\chi_\alpha \in \displaystyle\bigoplus_{|\beta| = n} H^\mu (U_\beta).$$

\subsection{First Return Map}

For $x \in U$, we write $r^n(x) = r(x) + r(f x) + \ldots + r(f^{n-1}x)$. In a three dimensional flow, the first return map $r$ is always Lipschitz continuous on 2-cylinders $U_{i,j}$ for any $i, j \in {1,\ldots, k}$ with $A(i,j) = 1$. However in higher dimensions $r$ is not Lipschitz continuous in general, but it is always \Hol continuous on each $U_{i,j}$ \cite{Hasselblatt}. Henceforth $\mu$ refers to the largest exponent such that $r$ is $\mu$-\Hol on every $U_{i,j}$. Then $|r(x)| \leq |r|_\infty$ and $|r(x) - r(y)| \leq |r|_{\mu} d(x,y)^{\mu}$ for any $x,y$ in the same 2-cylinder.

\begin{lem} \label{r^m}
Whenever $x$ and $y$ are on the same cylinder $U_\alpha$ with $|\alpha| = m$, there exists a constant $B_1$ depending only on $r, C_1, \gamma$ and $\mu$ such that
$$|r^m(x) - r^m(y)| \leq B_1 d(f^{m-1} x,f^{m-1} y)^{\mu}.$$
\end{lem}

\begin{proof}
\begin{align*}
|r^m(x) - r^m(y)| & \leq \displaystyle\sum_{j=0}^{m-1} |r(f^j x) - r(f^j y)|   \leq \displaystyle\sum_{j=0}^{m-1} |r|_\mu d(f^j x, f^j y)^{\mu} \\
			&\leq \displaystyle\sum_{j=0}^{m-1} |r|_{\mu} (\gamma^{m-j-1} C_1 d(f^{m-1} x,f^{m-1} y))^{\mu}\\
			&\leq |r|_{\mu} C_1^{\mu} d(f^{m-1} x,f^{m-1} y)^{\mu} \displaystyle\sum_{i=0}^{m-1} (\gamma^{\mu})^i \\
			&\leq |r|_{\mu} C_1^{\mu} \frac{1}{1-\gamma^{\mu}} d(f^{m-1} x,f^{m-1} y)^{\mu} \\
			&= B_1 d(f^{m-1} x,f^{m-1} y)^{\mu}
\end{align*}
\end{proof}




\subsection{The transfer operator and zeta function}

Let $s = a+ib \in \mathbb{C}$, with $|a| \leq a_0$ and $|b| \geq b_0$ for some constants $a_0, b_0$. Let $h$ denote the \textit{entropy} of the flow $\phi_t$, and for any continuous $w: U \rightarrow \mathbb{R}$ let $P(w)$ denote the \textit{pressure} of $w$. Definitions of these can be found in \cite{entropy}. We have $P(0) = h$.

Let $\hat{U} \subset U$ be the set of points $x \in \Int U$ such that $f^j x \in \Int U$ for any integer $j \geq 0$. We define the \textit{Ruelle transfer operator} $L_{-sr}: C(\hat{U}) \rightarrow C(\hat{U})$ by
$$L_{-sr}w(x) = \displaystyle\sum_{fy=x} e^{-sr(y)} w(y).$$

When $x \notin \hat{U}$, $r$ may not be continuous at $y$ for $f y = x$, and hence $L_{-sr}w$ may not be continuous at $x$. When $L_{-sr}w$ is \Hol continuous on $\hat{U}$, we extend this function to points in $U$ using the \Hol extension, by defining $L_{-sr}w(x) = \mbox{ext}(L_{-sr}w)(x)$. 

Define the zeta function $\zeta(s)$ for the flow $\phi_t$ by
$$\zeta(s) = \displaystyle\prod_\tau (1-e^{-s l(\tau)})^{-1},$$
where the product is over prime periodic orbits $\tau$, i.e. $f^n(x) = x, \forall x \in \tau$ for some $n$, and $l(\tau) = r^n(x)$ for some $x \in \tau$ is the length of the orbit. This zeta function converges to a nonzero analytic function for $\Re(s) > h$, where $h$ is the entropy of the flow $\phi_t$ \cite{Pollicott-Sharp}. It is important to choose $a_0$ large enough that the compact interval $I = [-a_0, a_0]$ contains $h$.

For all $s \in \mathbb{C}$, we define a weighted sum on the periodic points
$$Z_n(-sr) = \displaystyle\sum_{f^n(x)=x} e^{-sr^n(x)}.$$

We can use the fact that
$$\zeta(s) = \exp \displaystyle\sum_{n=1}^\infty \frac{1}{n} Z_n(-sr)$$
(see \cite{Pollicott-Sharp}) to relate the zeta function to the Ruelle transfer operator $L_{-sr}$, using Ruelle's Lemma, which relates $Z_n(-sr)$ to the Ruelle operator $L_{-sr}$.

\section{Ruelle's Lemma}

\begin{thm} (Ruelle's Lemma) \label{main lemma}
For each Markov leaf $U_i$, fix an arbitrary point $x_i \in U_i$. For all $\varepsilon >0, a_0 > 0$ and $b_0 > 0$, there exists $C_\varepsilon >0$ such that

$$
\left|Z_n(-sr) - \displaystyle\sum_{i=1}^k L_{-sr}^n \chi_i (x_i) \right| \leq C_\varepsilon |\Im(s)| \sum_{m=2}^n \left(\|L^{n-m}_{-sr}\|_\mu (\gamma^\mu)^m e^{m (\varepsilon + P(-ar))} \right),
$$
for all $s = a + i b \in \mathbb{C}$ with $|a| \leq a_0$ and $|b| \geq b_0$. Here $\mu$ denotes the \Hol exponent of $r$ and $P(-ar)$ denotes the pressure of $-ar$.

\end{thm}

Our aim in this article is to give a rigorous proof of this lemma. 

\subsection{Proof}

Fix an arbitrary point $x_i \in U_i$ for each Markov leaf $U_i$. Let $\varepsilon > 0, a_0 > 0$ and $b_0 > 0$.


\begin{prop} Let $m \geq 1$. Given a function $w$ that is $\mu$-\Hol on all cylinders of length $m+1$, the transfer operator of $w$ is always $\mu$-\Hol on all $m$-cylinders. That is,
$$L_{-s r}: \displaystyle\bigoplus_{|\alpha| = m+1} H^\mu (U_\alpha) \rightarrow \displaystyle\bigoplus_{|\alpha| = m} H^\mu (U_\alpha).$$

\begin{proof}
It suffices to show that for any particular $\alpha$, if $w$ is $\mu$-\Hol on $U_{i \alpha}$ for all $i$ with $A(i, \alpha_0)=1$, then $L_{-sr}w$ must be $\mu$-\Hol on $\Int U_\alpha$. 
Let $w \in H^\mu(U_{i \alpha})$ and let $x,y \in \Int U_\alpha$. Then using Lemma \ref{exp w}, we have

\begin{align*}
		&|L_{-sr} w(x) - L_{-sr} w(y)| = \left| \displaystyle\sum_{A(i,\alpha_0)=1} e^{-s r(ix)} w(ix) - \displaystyle\sum_{A(i,\alpha_0)=1} e^{-s r(iy)} w(iy) \right|\\ 
		&\leq \displaystyle\sum_{A(i,\alpha_0)=1} |e^{-sr(iy)}| \left(|e^{sr(iy) - sr(ix)}-1| |w(ix)| + |w(iy) - w(ix)|\right). \\
		&\leq \displaystyle\sum_{A(i,\alpha_0)=1} |e^{-sr(iy)}|  \big(|s||r|_\mu e^{a_0 |r|_\mu |U|^\mu} d(ix,iy)^\mu |w(ix)| + |w|_\mu d(ix,iy)^\mu  \big)\\
		&\leq k e^{a_0 |r|_\infty} \big(|s||r|_\mu e^{a_0 |r|_\mu |U|^\mu} |w|_\infty+ |w|_\mu \big) C_1^\mu \gamma^\mu d(x,y)^\mu \\ 
		&\leq C_4 \|w\|_\mu d(x,y)^\mu,
\end{align*}
for some constant $C_4$. So $L_{-sr} w$ is $\mu$-\Hol on $\Int U_\alpha$. Using the \Hol extension, $L_{-sr} w$ is $\mu$-\Hol on every $U_\alpha$.

\end{proof}
\end{prop}

\begin{cor} \label{corollary 1}
In particular, if $|\alpha| = m$ and $n \geq m$ then $L_{-sr}^n \chi_\alpha \in H^\mu(U_i)$ for all $i = 1, \ldots , k$.
\end{cor}

Since the map $f$ is expanding, any cylinder $U_\alpha$ can only contain zero or one $n$-periodic points, that is, a point $x \in U_\alpha$ such that $f^n x = x$ where $n = |\alpha|$. For each string $\alpha$, we have to choose a point $x_\alpha \in U_{\alpha}$ in a particular way. We choose a point from the cylinder as follows:

\begin{enumerate}[(1)] 
\item if $U_\alpha$ has an $n$-periodic point, then let $x_\alpha \in U_\alpha$ be such that $f^n x_\alpha = x_\alpha$;
\item if $U_\alpha$ has no $n$-periodic point and $n > 1$, then choose $x_\alpha \in U_\alpha$ arbitrarily such that $x_\alpha \notin f(U_{\alpha_{n-1}})$;
\item if $|\alpha|=n=1$, then $U_\alpha$ has no $n$-periodic point. Choose $x_\alpha = x_i$ (where $i=\alpha_0$ and $x_i \in U_i$ is one of the points we fixed from the start).
\end{enumerate}

Part of the second condition was introduced by Naud \cite{Naud}. Both \cite{Naud} and \cite{Pollicott-Sharp} leave out the last condition for $|\alpha|=1$, but \cite{Pollicott-Sharp 2} includes it. The need for this condition becomes clear in (\ref{short word}). Note that this is consistent with the other two, that is if $n=1$ and we choose $x_\alpha = x_i$, we still have that $x_\alpha \notin f(U_{\alpha_{n-1}})$ and $x_\alpha$ is not a 1-periodic point (i.e. a fixed point). This is because $A(i,i) = 0$ for all $i$, so there are no fixed points. 


\begin{lem} By choosing $x_\alpha$ in this way, we have 

\begin{displaymath}
   (L_{-sr}^n \chi_\alpha) (x_\alpha) = \left\{
     \begin{array}{lr}
       e^{-sr^n(x_\alpha)} & \emph{if $x_\alpha$ is a periodic point}, \\
       0 & \mathrm{otherwise}.
     \end{array}
   \right.
\end{displaymath} 

\begin{proof}
We have
$$(L_{-sr}^n \chi_\alpha) (x_\alpha) = \displaystyle \sum_{f^n y = x_\alpha} e^{-sr^n(y)} \chi_\alpha(y),$$
so we are looking for points $y$ such that $f^n y = x_\alpha$ and $y \in U_\alpha$.

Suppose $U_\alpha$ has an $n$-periodic point. Then there is exactly one point $y$ satisfying both $f^n y = x_\alpha$ and $\chi_\alpha(y) = 1$. Only $x_\alpha$ satisfies both of these, so we have $y = x_\alpha$. Thus $(L_{-sr}^n \chi_\alpha) (x_\alpha) = e^{-sr^n(x_\alpha)}$ when $U_\alpha$ has an $n$-periodic point.\\

Suppose $U_\alpha$ has no $n$-periodic point. Then $x_\alpha \in U_\alpha$ but $x_\alpha \notin f(U_{\alpha_{n-1}})$. Suppose for a contradiction that there is some $y \in U_\alpha$ s.t. $f^n y = x_\alpha$. Then 
$$y \in U_\alpha \Rightarrow y \in f^{-(n-1)} U_{\alpha_{n-1}} \Rightarrow f^n y \in f(U_{\alpha_{n-1}}) \Rightarrow x \in f(U_{\alpha_{n-1}}).$$
This is a contradiction, so the sum is empty. Thus $(L_{-sr}^n \chi_\alpha) (x_\alpha) = 0$ when $U_\alpha$ has no periodic point.

\end{proof}
\end{lem}

\begin{lem}
\label{Z}

$$Z_n(-sr) = \displaystyle \sum_{|\alpha|=n} (L^n_{-sr} \chi_\alpha)(x_\alpha)$$

\begin{proof}

Recall that $Z_n(-sr)$ is a sum over all period $n$ points of $U$. Since $U = \displaystyle \bigcup_{|\alpha|=n} U_\alpha$, we can break up this sum as follows

$$Z_n(-sr) = \displaystyle \sum_{f^n x = x} e^{-s r^n(x)} = \sum_{|\alpha|=n} \sum_{f^n x = x} e^{-s r^n(x)} \chi_\alpha(x).$$


Suppose $U_\alpha$ has an $n$-periodic point. Then it is equal to $x_\alpha$, so we have
$$\sum_{f^n x = x} e^{-s r^n(x)} \chi_\alpha(x) = e^{-s r^n(x_\alpha)} = (L^n_{-sr} \chi_\alpha)(x_\alpha).$$

Suppose $U_\alpha$ has no $n$-periodic points. Then we have
$$\sum_{f^n x = x} e^{-s r^n(x)} \chi_\alpha(x) = 0 = (L^n_{-sr} \chi_\alpha)(x_\alpha).$$

Then,
$$Z_n(-sr) = \displaystyle \sum_{|\alpha|=n} (L^n_{-sr} \chi_\alpha)(x_\alpha).$$

\end{proof}
\end{lem}

Note that, with the third condition we placed on $x_\alpha$, it is trivial to show that 
\begin{equation}
\displaystyle\sum_{i=1}^k (L_{-sr}^n \chi_{U_i})(x_i) = \sum_{|\alpha|=1} (L_{-sr}^n \chi_\alpha) (x_\alpha).
\label{short word}
\end{equation}

\begin{prop}
$$Z_n(-sr) - \displaystyle\sum_{i=1}^k L_{-sr}^n \chi_i (x_i)= \sum_{m=2}^n \left( \sum_{|\alpha|=m} L^n_{-sr} \chi_\alpha (x_\alpha) - \sum_{|\beta|=m-1} L^n_{-sr} \chi_\beta(x_\beta) \right)$$

\begin{proof}
Expand the sum on the RHS and notice that almost all the terms cancel out, leaving only
$$\displaystyle\sum_{|\alpha|=n} (L_{-sr}^n \chi_\alpha) (x_\alpha) - \sum_{|\alpha|=1} (L_{-sr}^n \chi_{U_\alpha})(x_\alpha).$$
Using Lemma \ref{Z} and (\ref{short word}), this is equal to
$$Z_n(-sr) - \displaystyle\sum_{i=1}^k L_{-sr}^n \chi_i (x_i).$$
\end{proof}
\end{prop}

Note that for any word $\beta$ of length $m-1$, we have $U_\beta = \bigcup_{A(\beta_{m-2},i) = 1} U_{\beta i}$, so for any $x \in M$ we can write 
$$\displaystyle \chi_\beta(x) = \sum_{A(\beta_{m-2},i) = 1} \chi_{\beta i}(x).$$

So we have, for $m \leq n$,
$$\displaystyle \sum_{|\beta|=m-1} L^n_{-sr} \chi_\beta(x_\beta) = \sum_{|\beta|=m-1} \sum_{A(\beta_{m-2},i) = 1} L^n_{-sr} \chi_{\beta i} (x_\beta) = \sum_{|\alpha|=m} L^n_{-sr} \chi_\alpha (x_{\overline{\alpha}}),$$
and therefore
$$Z_n(-sr) - \displaystyle\sum_{i=1}^k L_{-sr}^n \chi_i (x_i)= \sum_{m=2}^n \sum_{|\alpha|=m} \left[ L^n_{-sr} \chi_\alpha (x_\alpha) - L^n_{-sr} \chi_\alpha (x_{\overline{\alpha}}) \right].$$

By Corollary \ref{corollary 1}, $L^n_{-sr} \chi_\alpha$ is $\mu$-\Hol everywhere on $U$. So we have, for all $n \geq 2$,

\begin{align*} 
\left| L^n_{-sr} \chi_\alpha (x_\alpha) - L^n_{-sr} \chi_\alpha (x_{\overline{\alpha}}) \right|
  & \leq \|L^n_{-sr} \chi_\alpha\|_\mu \hspace{2pt} d(x_\alpha, x_{\overline{\alpha}})^\mu \\
  & \leq \|L^{n-m}_{-sr}\|_\mu . \|L^m_{-sr} \chi_\alpha\|_\mu \hspace{2pt} d(x_\alpha, x_{\overline{\alpha}})^\mu,
\end{align*}
where $\| \hspace{4pt} \|_\mu$ is the operator norm derived from the $\mu$-\Hol norm. 

So we have the following estimate for all $n \geq 2$:
\begin{equation} \label{pre-lemma}
\left|Z_n(-sr) - \displaystyle\sum_{i=1}^k L_{-sr}^n \chi_i (x_i) \right| \leq \sum_{m=2}^n \sum_{|\alpha|=m} \|L^{n-m}_{-sr}\|_\mu . \|L^m_{-sr} \chi_\alpha\|_\mu \hspace{2pt} d(x_\alpha, x_{\overline{\alpha}}).
\end{equation}

\subsection{Particular estimates}

The last part of the proof is to find particular estimates for the three parts in the RHS of equation \ref{pre-lemma}. 

\begin{enumerate}

	\item In the low dimensional case, $\mu=1$, so $\|L_{-sr}^{n-m}\|_1$ can sometimes be estimated using an important theorem of Dolgopyat \cite{Pollicott-Sharp}. In this proof we leave it as $\|L_{-sr}^{n-m}\|_\mu$. Later we will show an example of how Dolgopyat's theorem can be used in some cases along with the lemma, to get an extension of the zeta function.	
	\item $\|L_{-sr}^m \chi_\alpha\|_\mu$ is simply estimated by $\|L_{-sr}^m \chi_\alpha\|_\mu \leq C |\Im(s)|$ in the Pollicott-Sharp paper \cite{Pollicott-Sharp}. Naud produces a more detailed estimate that we cover in detail in the next section.
	\item Since $x_\alpha, x_{\overline{\alpha}}$ are in the same cylinder $U_{\overline{\alpha}}$, (\ref{expanding map}) gives us 
	\begin{equation}
	d(x_\alpha, x_{\overline{\alpha}})^\mu \leq C_1^\mu \gamma^{\mu(m-2)} d( f^{m-2} x_\alpha, f^{m-2} x_{\overline{\alpha}} )^\mu \leq C_5 (\gamma^\mu)^m,
	\end{equation}
	for some constant $C_5>0$. 
	
\end{enumerate}

We now fix $m \geq 1$ and attempt to estimate $\|L_{-sr}^m(\chi_\beta)\|_\mu$.

\begin{lem}

\begin{displaymath}
   (L_{-sr}^m \chi_\beta) (x) = \left\{
     \begin{array}{lr}
       e^{-r^m(f_\beta^{-1}x)} & \emph{if $x \in f(U_{\beta_{m-1}})$}, \\
       0 & \mathrm{otherwise}.
     \end{array}
   \right.
\end{displaymath} 

\begin{proof}
We have $(L_{-sr}^m \chi_\beta) (x) = \displaystyle\sum_{f^my=x} e^{-sr^m(y)} \chi_\beta (y)$. 

Suppose $x \in f(U_{\beta_{m-1}})$. Then $f^my=x$ and $\chi_\beta (y)=1$ implies $y = f_\beta^{-1}(x)$. So there is only one non-zero term in this sum, which is $e^{-r^m(f_\beta^{-1}x)}$.

Now suppose $x \notin f(U_{\beta_{m-1}})$. Then if $f^my=x$ and $\chi_\beta (y)=1$, then $f^{m-1}y \in U_{\beta_{m-1}}$, so $x = f^m(y) \in f(U_{\beta_{m-1}})$, which is a contradiction. So the sum is zero.

\end{proof}

\label{L of chi}
\end{lem}



For each admissible word $\beta$ with $|\beta| = m$, we fix a point $y_\beta \in f(U_{\beta_{m-1}})$. We will see later how to choose $y_\beta$. Set $z_\beta = f_\beta^{-1}(y_\beta)$.

\begin{lem}\label{total norm of L}
$$\|L_{-sr}^m(\chi_\beta)\|_\mu \leq (e^{a_0 |U|^\mu B_1} + B_1 |s| e^{a_0 |U|^\mu (1 + \gamma^\mu)B_1}) e^{-a r^m(z_\beta)}.$$
\end{lem}

\begin{proof}
Let $|U| = \displaystyle\max_i \mathrm{diam}(U_i)$ be the largest diameter possible for a Markov leaf. Let $x,y$ be in the same Markov leaf. If $x,y \notin f(U_{\beta_{m-1}})$, then clearly $|L_{-sr}^m(\chi_\beta)(x)| = |L_{-sr}^m(\chi_\beta)(x) - L_{-sr}^m(\chi_\beta)(y)| = 0$. Otherwise, both $x$ and $y$ are in $f(U_{\beta_{m-1}})$. 

$x$ and $y_\beta$ may not be on the same leaf, but $f_\beta^{-1} x$ and $f_\beta^{-1} y_\beta$ are on the same cylinder $U_\beta$. 
So using Lemmas \ref{L of chi} and \ref{r^m}, we have 

\begin{align*}|L_{-sr}^m \chi_\beta(x)| & = |e^{-(a+ib) r^m(f_\beta^{-1} x)}|  = e^{-a r^m(f_\beta^{-1} x)}\\
 													&\leq e^{|a| |r^m(f_\beta^{-1} x) - r^m(f_\beta^{-1} y_\beta)|} \hspace{4pt} e^{-ar^m(f_\beta^{-1} y_\beta)}\\
 													&\leq e^{a_0 B_1 d(f^{m-1} f_\beta^{-1} x, f^{m-1} f_\beta^{-1} y_\beta)^\mu} \hspace{4pt} e^{-a r^m(z_\beta)} \leq e^{a_0 |U|^\mu B_1} \hspace{4pt} e^{-ar^m(z_\beta)}.
\end{align*}

Since $f_\beta^{-1} x$ and $f_\beta^{-1} y$ are on the same cylinder $U_\beta$, and $x,y$ are on the same leaf, we have $|r^m(f_\beta^{-1} x) - r^m(f_\beta^{-1} y)| \leq \gamma^\mu B_1 d(x,y)$. Using Lemma \ref{exp w}, we get

\begin{align*}
|L_{-sr}^m(\chi_\beta)(x) - L_{-sr}^m(\chi_\beta)(y)|	&\leq |e^{-s r^m(f_\beta^{-1} x) + s r^m(f_\beta^{-1} y)} - 1|  |e^{-s r^m(f_\beta^{-1} y)}|\\
																											&\leq \left(|s| B_1 \gamma^\mu d(x,y)^\mu e^{|a| B_1 \gamma^\mu |U|^\mu } \right)  \left( e^{a_0 B_1 |U|^\mu }e^{-ar^m(z_\beta)}\right)\\
																											&\leq B_1 \gamma^\mu |s| e^{a_0 B_1 (\gamma^\mu +1) |U|^\mu} e^{-ar^m(z_\beta)} d(x,y)^\mu.
\end{align*}

Combining these results we get 
\begin{align*}
\|L_{-sr}^m(\chi_\beta)\|_\mu &= |L_{-sr}^m(\chi_\beta)|_\infty + |L_{-sr}^m(\chi_\beta)|_\mu \\
&\leq e^{a_0 |U|^\mu B_1} e^{-a r^m(z_\beta)} + B_1 \gamma^\mu |s| e^{a_0 |U|^\mu (1+\gamma^\mu) B_1}e^{-a r^m(z_\beta)}\\
&\leq (e^{a_0 |U|^\mu B_1} + B_1 \gamma^\mu |s| e^{a_0 |U|^\mu (1+\gamma^\mu) B_1}) e^{-a r^m(z_\beta)}.
\end{align*}

\end{proof}

Now $|s| = |a + b i| \leq |a| + |b|$, where $|a| \leq a_0$ and $|b| = |\Im(s)| \geq b_0$. So there exists a constant $B_2$ such that

$$\displaystyle\sum_{|\beta|=m} \|L_{-sr}^m(\chi_\beta)\|_\mu \leq  B_2 |\Im(s)| \sum_{|\beta|=m} e^{-a r^m(z_\beta)}.$$

We now estimate the sum on the right. Naud estimates it directly, using a well known theorem \cite{sum convergence, Bowen}. However in \cite{sum convergence} this theorem does not involve a sum over words $|\beta|=m$, but a sum over $m$-periodic points $f^m z = z$. In \cite{Bowen}, it's not clear that the theorem holds in the way that Naud claims either, so we show how to recover the estimate another way using the theorem as stated in \cite{sum convergence}. Define a function 

$$\varphi_m(a) = \left( \displaystyle \sum_{f^m z = z} e^{-a r^m(z)} \right)^\frac{1}{m}.$$

\begin{thm} (\cite{sum convergence})
Let $r: U \rightarrow \mathbb{R}$ be the first return map corresponding to the flow $\phi_t$, let $f: U \rightarrow U$ be the expanding map corresponding to the flow and let $a \in I \subset \mathbb{R}$ where $I$ is a closed, bounded set in $\mathbb{R}$ containing $h$, the entropy of the flow. Let $P(-ar)$ be the \textit{pressure} of $-ar$. Then we have $\varphi_m(a) \rightarrow e^{P(-ar)}$ uniformly in $a \in I$ as $m \rightarrow \infty$. Hence, 

$$\ln \varphi_m(a) = \frac{1}{m} \ln \displaystyle \sum_{f^m z = z} e^{-a r^m(z)} \rightarrow P(-ar) \mbox{ as } m \rightarrow \infty,$$
uniformly in $a \in I$.
\end{thm}

Therefore for all $\varepsilon >0$ there exists $N$ such that for all $m \geq N$ and $a \in I$, we have $|\ln \varphi_m(a) - P(-ar)| < \varepsilon$. This gives us 
$$\ln \varphi_m(a) \leq |\ln \varphi_m(a) - P(-ar)| + P(-ar) < \varepsilon + P(-ar),$$
which implies that $\varphi_m(a) \leq e^{\varepsilon + P(-ar)}$ for all $m \geq N$ and $a \in I$.

\begin{thm} \label{B epsilon}
For all $\varepsilon >0$ there exists $B_\varepsilon > 0$ such that for any real $a$ in $I$, and for all $m \geq 1$ we have
$$\left( \displaystyle \sum_{f^m z = z} e^{-a r^m(z)} \right) \leq B_\varepsilon e^{m(\varepsilon + P(-ar))}.$$

\begin{proof} 
Let $b_\varepsilon = \max \{\varphi_m(a)^m : 0 < m < N, a \in A \}$, and let $B_\varepsilon = \max\{b_\varepsilon,1\}$. Then the inequality holds.

\end{proof}
\end{thm}

Recall that we can choose $y_\beta$, and hence $z_\beta \in U_\beta$ however we like. Since $U_\beta$ may not have an $m$-periodic point, we cannot simply choose $f^m z_\beta = z_\beta$ and use the above theorem. However, $U_\beta$ must have a periodic point of some higher order, because the periodic points of $f$ are dense in $U$. We need a much larger constant to get the required inequality.

\begin{defn}
Let $p(\beta)$ be the smallest integer such that $U_\beta$ has a $p(\beta)$ periodic point $z_\beta$ i.e. $f^p(\beta) z_\beta = z_\beta$. Define $p(m)$ for any $m$ to be the smallest integer such that for all words $\beta$ of length $m$, there exists $z_\beta \in U_\beta$ with $f^{p(m)} z_\beta = z_\beta$. Equivalently, we define 
$$p(m) = \mbox{lcm}\{p(\beta) : |\beta| = m \},$$
where lcm is the lowest common multiple of the set.
\end{defn}

Now if we choose $z_\beta$ so that $f^{p(m)} z_\beta = z_\beta$, then we have 
$$\displaystyle\sum_{|\beta|=m} e^{-a r^{p(m)}(z_\beta)} = \sum_{f^{p(m)}z = z} e^{-a r^{p(m)}(z)}.$$

Finally we write $e^{-a r^m(z_\beta)}$ in terms of $e^{-a r^{p(m)}(z_\beta)}$ to get the estimate we need. For any $m \geq 1$, we have $p(m) \leq m+d$ for some integer constant $d$ that depends only on the matrix $A$. This clearly follows from the irreducibility of $A$. We have
\begin{align*}
r^{p(m)}(z_\beta) &\leq r^{m+d}(z_\beta) = r(z_\beta) + \ldots + r(f^{m+d-1} z_\beta)\\
 &\leq  r^m(z_\beta) + r(f^m z_\beta) + \ldots + r(f^{m+d-1} z_\beta),\\
|r^{p(m)}(z_\beta) - r^m(z_\beta)| &\leq r^d(f^m z_\beta) \leq d |r|_\infty.
\end{align*}

So,
\begin{align*}
\displaystyle\sum_{|\beta|=m} e^{-a r^m(z_\beta)}	&\leq \displaystyle\sum_{|\beta|=m} e^{-ar^{p(m)}(z_\beta)} |e^{-a (r^m(z_\beta) - r^{p(m)}(z_\beta))}|\\
																									&\leq \displaystyle\sum_{|\beta|=m} e^{-ar^{p(m)}(z_\beta)} e^{|a| d |r|_\infty}\\
																									&\leq e^{d |a| |r|_\infty} \displaystyle\sum_{f^{p(m)}z = z} e^{-ar^{p(m)}(z_\beta)}.
\end{align*}

Now we can finally use Theorem \ref{B epsilon}. For all $\varepsilon > 0$, there exists $B_\varepsilon >0$ such that
$$\displaystyle\sum_{|\beta|=m} e^{-a r^m(z_\beta)}	\leq e^{d |a| |r|_\infty} B_\varepsilon e^{(m+d) (\varepsilon + P(-ar))}.$$

We can group the constants together by defining $C_\varepsilon = C_5 B_2 B_\varepsilon e^{d |a| |r|_\infty} e^{d(\varepsilon + P(-ar))}$. Then we get
$$\displaystyle\sum_{|\beta|=m} \|L_{-sr}^m(\chi_\beta)\|_\mu \leq C_5^{-1}C_\varepsilon |\Im(s)| e^{m (\varepsilon + P(-ar))}.$$

Combining this with the third particular estimate, and the inequality (\ref{pre-lemma}), we get the following:

For any set of points $x_1, \ldots, x_k \in U_1, \ldots, U_k$, and for all $\varepsilon > 0$, there exists a constant $C_\varepsilon > 0$ such that,

$$\left|Z_n(-sr) - \displaystyle\sum_{i=1}^k L_{-sr}^n \chi_i (x_i) \right| \leq C_\varepsilon |\Im(s)| \sum_{m=2}^n \left(\|L^{n-m}_{-sr}\|_\mu (\gamma^\mu)^m e^{m (\varepsilon + P(-ar))} \right).$$ 
and this proves Ruelle's Lemma.

\section{Applying Ruelle's Lemma}

Here we explain some of the context of Ruelle's Lemma, and show how it can be used in some cases to estimate the Ruelle zeta function. This was done by Pollicott and Sharp in \cite{Pollicott-Sharp, Pollicott-Sharp 2} and later by Naud \cite{Naud} in different cases. Both estimates apply in less general cases than the conditions for Theorem \ref{main lemma}. In particular they only apply when the stable and unstable manifolds are one-dimensional. Both estimates also depend on special cases of a very important result of Dolgopyat \cite{Dolgopyat 1, Dolgopyat 2}. As an example of how Theorem \ref{main lemma} can be used, we show an estimate of the zeta function in the case considered by Naud \cite{Naud}, which uses the following version of Dolgopyat's theorem.

\begin{thm}(Dolgopyat) \label{Dolgopyat thm}
Let $\phi_t$ be a smooth transitive Anosov flow on a compact Riemannian manifold such that $\phi_t$ is non-integrable and has entropy h, and the local stable and unstable foliations are $C^1$. Then there exists $C_6>0, \varepsilon_2 > 0, b_0 > 0$ and $0 < \rho < 1$ such that whenever $h - \varepsilon_2 \leq \Re(s) \leq h$ and $|\Im(s)| \geq b_0$,
$$\|L_{-sr}^n\|_\Lip \leq C_6 |\Im(s)| \rho^n, \hspace{4pt} \forall n \geq 0.$$
\end{thm}

A similar estimate for open billiard flows was proved in \cite{Stoyanov 1}, and more general results were proved in \cite{Stoyanov 3}. Since $\Stab{x}$ and $\Unstab{x}$ are $C^1$, the first return map is Lipschitz so we have $\mu = 1$.\\

Let $a = \Re(s) \in [h - \varepsilon_2, h]$ and let $|\Im(s)| \geq b_0$ as in the above theorem. By applying theorems \ref{main lemma} and \ref{Dolgopyat thm}, we get

\begin{align*}
|Z_n(s)|	&\leq |Z_n(s) - \displaystyle\sum_{i=1}^k L_{-sr}^n(\chi_i)(x_i) | + |\sum_{i=1}^k L_{-sr}^n(\chi_i)(x_i)|\\
					&\leq C_\varepsilon |\Im(s)| \sum_{m=2}^n \left( \|L^{n-m}_{-sr}\|_\Lip \gamma^m e^{m (\varepsilon + P(-ar))}\right) + \sum_{i=1}^k \|L_{-sr}^n\|_\Lip \\
					&\leq C_\varepsilon |\Im(s)|  \sum_{m=2}^n \left( C_6 |\Im(s)| \rho^{n-m} \gamma^m e^{m (\varepsilon + P(-ar))}\right) + k C_6 |\Im(s)| \rho^n\\
					&\leq C_6 C_\varepsilon |\Im(s)|^2 \rho^n \sum_{m=2}^n \left(\frac{\gamma}{\rho} e^{\varepsilon + P(-ar)}\right)^m + k C_6 |\Im(s)| \rho^n.
\end{align*}

Although the $\rho$ in Theorem \ref{Dolgopyat thm} is fixed, the inequality in the theorem still holds for larger $\rho$, as long as $\rho < 1$. Thus we can choose $\rho$ so that $1 > \rho > \gamma$. Note that $P(-h r) =0$, so by taking $\varepsilon$ small enough, $a = \Re(s)$ sufficiently close to $h$, say $|a - h| < \delta$ for some $\delta > 0$, and $\frac{\gamma}{\rho}$ sufficiently small, we can have 
$$\frac{\gamma}{\rho} e^{\varepsilon + P(-ar)} < 1.$$
This simplifies the inequality to
$$|Z_n(s)| 	\leq C_6 C_\varepsilon |\Im(s)|^2 \rho^n + k C_6 |\Im(s)| \rho^n \leq C_7 |\Im(s)|^2  \rho^n,$$					
where $C_7 = C_6 C_\varepsilon + k C_6 b_0^{-1}$ is simply another constant. Now we can obtain an estimate for $\zeta(s)$. The Ruelle zeta function \cite{Pollicott-Sharp} is given by

$$\zeta(s) = \exp \left( \displaystyle\sum_{n=1}^\infty \frac{1}{n} Z_n(-sr) \right).$$ 
So we can continue this zeta function analytically to the domain 
$$\{s: |\Re(s) - h| < \min\{\varepsilon_2, \delta\} \mbox{ and } |\Im(s)| \geq b_0 \},$$ 
for some $\varepsilon_2, b_0 > 0$. In this region we have 

$$\exp \left( -\displaystyle\sum_{n=1}^\infty \frac{1}{n} |Z_n(-sr)| \right) \leq |\zeta(s)| \leq \exp \left( \displaystyle\sum_{n=1}^\infty \frac{1}{n} |Z_n(-sr)| \right)$$ 
$$\exp \left( -C_7 |\Im(s)|^2 \displaystyle\sum_{n=1}^\infty \frac{1}{n}  \rho^n \right) \leq |\zeta(s)| \leq \exp \left( C_7 |\Im(s)|^2 \sum_{n=1}^\infty \frac{1}{n}  \rho^n \right)$$ 
$$\exp \left( -C_7 |\Im(s)|^2 \ln\left(\frac{1}{1-\rho}\right) \right) \leq |\zeta(s)| \leq \exp \left( C_7 |\Im(s)|^2 \ln\left(\frac{1}{1-\rho}\right) \right).$$ \\

This demonstrates the application of Ruelle's Lemma to estimating the zeta function.

\footnotesize{UNIVERSITY OF WESTERN AUSTRALIA, CRAWLEY WA, 6009, AUSTRALIA}\\
\footnotesize{\textit{E-mail address: paul.e.wright@uwa.edu.au}}

\end{document}